\documentclass{amsart}

\usepackage[utf8]{inputenc}
\usepackage{amsmath,amsfonts,amsthm,amssymb,verbatim,etoolbox,color,enumerate,tikz,float}
\usepackage{hyperref}

\usepackage[most]{tcolorbox}

\newcommand{\bbn}{\mathbb{N}}

\newcommand{\bbr}{\mathbb{R}}

\newcommand{\abs}[1]{\left\lvert #1\right\rvert}

\newcommand{\brac}[1]{\left( #1\right)}

\newcommand*{\bbe}{
  \mathop{
    \mathchoice{\vcenter{\hbox{\larger[4]$\mathbb{E}$}}}
               {\kern0pt\mathbb{E}}
               {\kern0pt\mathbb{E}}
               {\kern0pt\mathbb{E}}
  }\displaylimits
}

\providecommand{\floor}[1]{\left\lfloor#1\right\rfloor}

\newtheorem{theorem}{Theorem}
\newtheorem{lemma}[theorem]{Lemma}

\newtheorem{corollary}[theorem]{Corollary}
\newtheorem{conjecture}[theorem]{Conjecture}

\theoremstyle{definition}

\newtheorem{remark}[theorem]{Remark}

\begin{document}
 \title{Additive structure in convex sets}
\author{Thomas F. Bloom}
\address{Department of Mathematics, University of Manchester, Manchester, M13 9PL}
\email{thomas.bloom@manchester.ac.uk}
\author{Jakob F\"{u}hrer}
\address{Institute for Algebra, Johannes Kepler Universit\"{a}t\\
Linz, Austria}
\email{jakob.fuehrer@jku.at}
\author{Oliver Roche-Newton}
\address{Institute for Algebra, Johannes Kepler Universit\"{a}t\\
Linz, Austria}
\email{o.rochenewton@gmail.com}

\begin{abstract}
    This paper considers some different measures for how additively structured a convex set can be. The main result gives a construction of a convex set $A$ containing $\Omega(|A|^{3/2})$ three-term arithmetic progressions.
\end{abstract}

\maketitle

\section{Introduction}

%\textcolor{blue}{ORN: I started writing an introduction but ran out of time and now need to go on the Kindergarten run. I hope to have another window tomorrow to do some more. I suggest you both ignore it for now.}

A convex set is a finite set $A\subset \bbr$ such that, if the elements are ordered $A=\{a_1<\cdots<a_n\}$, the consecutive differences
\[d_i=a_{i+1}-a_i\quad\textrm{ for }\quad 1 \leq i \leq n-1\]
form a strictly increasing sequence. This is equivalent to the statement that $A=\{ f(x) : 1\leq x\leq n\}$ for a strictly convex and increasing function $f:\bbr \to \bbr$. In a similar spirit to the sum-product problem, one expects that a strictly convex function disturbs (or even destroys) the additive structure of a set or, to put this another way, that a convex set cannot be additively structured. 

This vague principle can be quantified in various ways. It was conjectured by Erd\H{o}s  %\textcolor{blue}{ORN: I dug around a bit more and not been able to find a reference to the original conjecture yet. Several papers say it is a conjecture of Erd\H{o}s but give no reference.} 
that the bound\footnote{Here and throughout this paper, the notation  $X\gg Y$, $Y \ll X,$ $X=\Omega(Y)$, and $Y=O(X)$ are all equivalent and mean that $X\geq cY$ for some absolute constant $c>0$. The notation $X \asymp Y$ (or equivalently $X= \Theta(Y)$) indicates that both $X \ll Y$ and $X \gg Y$ hold.}
\[
|A+A| \gg |A|^{2-o(1)}
\]
holds for any convex set $A$, with a similar conjecture for the difference set $A-A$. This problem remains open, with the current best-known bounds
\[
|A+A| \gg |A|^{30/19-o(1)} ,\,\,\,\,\, |A-A| \gg |A|^{6681/4175-o(1)}
\]
due to Rudnev and Stevens \cite{RuSt22} and Bloom \cite{Bl25} respectively. Note that the latter bound gives a small improvement on the previous record with exponent $8/5-o(1)$, due to Schoen and Shkredov \cite{ScSh11}.

Another useful way to measure of the additive structure of a set is the additive energy, defined by
\[
E(A):= | \{ (a,b,c,d) \in A^4: a+b=c+d \}|.
\]
Since a convex set should be additively unstructured, one expects $E(A)$ to be small for any convex $A$. The trivial solutions $a+b=a+b$ already give $E(A) \geq |A|^2$, and the example of the first $n$ squares shows that the energy can be slightly larger than quadratic by a logarithmic factor. The best-known bound
\begin{equation} \label{energyupper}
E(A) \ll |A|^{\frac{123}{50}+o(1)}
\end{equation}
is due to Bloom \cite{Bl25}, slightly improving an earlier bound of Shkredov \cite{Sh13}. There are no known constructions of large convex sets $A$ such that $E(A) \geq |A|^{2+c}$ for a strictly positive $c$. With this in mind, the following rather strong conjecture is not entirely implausible, although we do approach it with some suspicion.
\begin{conjecture}\label{conj1}
If $A\subset \bbr$ is convex then
\[E(A) \leq \abs{A}^{2+o(1)}.\]
\end{conjecture}

%The first problem that we consider in this paper is similar to the energy problem; how many three-term arithmetic progressions (henceforth will abbreviate to $3$APs) can a convex set have? 
%Sidon sets in convex sets

One can also consider solutions to other additive equations, and a close relative of the energy $E(A)$ is the quantity $T_3(A)$, which counts the number of three-term arithmetic progressions inside $A$. That is, define 
\[
T_3(A):= |\{ (a,b,c) \in A^3 : 2a=b+c \}|.
\]
Trivially we have
\[
\abs{A}\leq T_3(A)\leq \abs{A}^2.
\]
The lower bound comes from considering the trivial arithmetic progressions whereby $a=b=c$, while the upper bound follows from the observation that fixing $a$ and $b$ determines $c$. Since very structured sets (for example, arithmetic progressions) have $T_3(A)\gg \abs{A}^2$, the count $T_3(A)$ can be viewed as an alternative measure of additive structure. To this end, it is natural to ask how it behaves for convex sets. Of course, one can take convex sets which do not contain any non-trivial three-term arithmetic progressions (for example, $A=\{1,2,4,\ldots,2^n\}$), and hence $T_3(A)=\abs{A}$ is certainly possible. The interesting question is how large $T_3(A)$ can possibly be. In this paper, perhaps surprisingly considering the above heuristic that convex sets lack additive structure, we construct a convex set containing many $3$-term arithmetic progressions.

\begin{theorem} \label{thm:construction1}
There are arbitrarily large convex sets $A\subset \bbr$ which contain $\gg \abs{A}^{3/2}$ many non-trivial $3$-term arithmetic progressions.   
\end{theorem}

Moreover, our construction generalises to give convex sets which have relatively many copies of various additive configurations, including longer progressions. See Theorem \ref{thm:constructiongen} for a more general form of the statement. The construction appeals to an idea from a paper of Ruzsa and Zhelezov \cite{RuZh19} in which they proved the existence of a set $A$ with the property that $A+A$ contains a large convex set (see also Bhowmick, Lund, and Roche-Newton \cite{BLR24}).

For upper bounds, we observe that Conjecture \ref{conj1} would imply a nearly matching upper bound $T_3(A) \ll |A|^{3/2+o(1)}$. Indeed, since $E(A)=\sum_x r_{A+A}(x)^2$, the Cauchy-Schwarz inequality yields
\[
T_3(A) =\sum_{a\in A}r_{A+A}(2a)\leq \abs{A}^{1/2}E(A)^{1/2}\leq \abs{A}^{3/2+o(1)}.
\]
To put this observation in its contrapositive form, any construction with $T_3(A) \geq |A|^{3/2}$ would also give a construction refuting Conjecture \ref{conj1}. As an unconditional upper bound, we have the following result.
\begin{theorem} \label{thm:3APupper}
  If $A \subset \mathbb R$ is convex then
  \[
  T_3(A) \ll |A|^{5/3}.
  \]
\end{theorem}

Theorem \ref{thm:3APupper} is an immediate corollary of an upper bound for the number of representations of an element of $A+A$. Let $r_{A+A}(x)$ denote the number of such representations, that is,
\[
r_{A+A}(x)= | \{ (a,b) \in A \times A : a+b=x \}|.
\]
Schoen \cite{Sc14} proved that the bound $r_{A+A}(x) \ll |A|^{2/3}$ holds for any convex set $A$ and any $x \in \mathbb R$. Theorem~\ref{thm:3APupper} follows since
\[T_3(A)=\sum_{a\in A}r_{A+A}(2a).\]
In this paper we give an alternative proof of Schoen's upper bound, and also give a construction of a convex set such that
\[
r_{A+A}(x) \gg |A|^{2/3},
\]
thus showing that Schoen's upper bound for additive representations is optimal.

%\textcolor{blue}{Could possibly already say something here about the connection with Jarnik curve theorem and the sense that this result is partially known, but I guess better to put this historical context in the section where we give the construction.}

We also consider the problem of finding a large Sidon set in a convex set. A set is Sidon if  the only solutions to the energy equation $a+b=c+d$ are the trivial ones with $\{a,b\}= \{c,d \}$. Another way to measure the additive structure of the set is to determine the size of the largest Sidon set contained within it. Given an arbitrary finite set $A \subset \mathbb R$, let $S(A)$ denote this quantity, that is,
\[
S(A)= \max \{ |B| : B \text { is Sidon and }B \subset A \}.
\]
It follows from a more general result of Koml\'{o}s, Sulyok, and Szemer\'{e}di \cite{KSS75} that $S(A) \gg |A|^{1/2}$ for any finite $A \subset \mathbb R$, and one can see that this is optimal by considering the very additively structured set $A= \{1,\dots, n \}$. (There are several ways to construct Sidon sets of size $\gg n^{1/2}$ in $\{1,\ldots,n\}$. The first such construction is due to Singer \cite{Si38}.) Intuitively, we expect that less additively structured sets should always contain larger Sidon sets. With this in mind, we consider the following problem; given a convex set $A \subset \mathbb R$, how small can $S(A)$ be? The next two statements give some partial answers to this question.

\begin{theorem} \label{thm:Sidonlower}
If $A\subset \bbr$ is convex then
    \[
    S(A) \gg |A|^{\frac{1}{2}+\frac{1}{75}-o(1)}.
    \]
\end{theorem}

%The proof uses the energy bound \eqref{energyupper} as an input. If we had an optimal energy bound as given by Conjecture~\ref{conj1} then our proof would deliver $S(A)\gg \abs{A}^{2/3-o(1)}$. 
From the other side, we obtain the following result.

\begin{theorem} \label{thm:Sidonupper}
    There exist arbitrarily large convex sets $A\subset \bbr$ such that 
    \[
    S(A) \ll |A|^{3/4}.
    \]
\end{theorem}

This question has a similar flavour to a Sidon set analogue of the sum-product problem posed by Klurman and Pohoata \cite{Po21}. They conjectured that any set $A \subset \mathbb R$ must contain either a large additive Sidon set or a large multiplicative Sidon set (i.e. a large set containing only trivial solutions to the equation $ab=cd$). The proof of Theorem \ref{thm:Sidonupper} is a fairly straightforward consequence of the result of Ruzsa and Zhelezov \cite{RuZh19}, and uses similar ideas Roche-Newton and Warren \cite{RoWa21} used to refute a strong form of the Klurman-Pohoata Conjecture. We refer to Shkredov \cite{Sh23} for more details on this Sidon sum-product problem.

%Finally, it is worth mentioning that the constructions given in the proofs of Theorems \ref{thm:construction1} and \ref{thm:Sidonupper} make use of an idea that Ruzsa and Zhelezov \cite{RuZh19}

%\textcolor{blue}{ORN note to self: finish the introduction here by mentioning the paper of Ruzsa and Zhelezov as it is a crucial connecting idea in both of our constructions.}

We conclude the introduction by summarising what we know and believe about various structures within convex sets.
\begin{itemize}
\item \emph{Three-term arithmetic progressions}: %Let $T(A)$ count the number of three-term arithmetic progressions in $A$. 
We know that, if $A$ is a convex set, then
\[\abs{A}\leq T_3(A)\ll \abs{A}^{5/3}.\]
The existence of convex sets with $T_3(A)=\abs{A}$ is trivial (for example, take powers of $2$). In this paper we construct a convex set $A$ such that $T(A)\gg \abs{A}^{3/2}$ and we believe that
\[T(A) \ll\abs{A}^{3/2+o(1)}\]
for all convex $A$. 
\item \emph{Sidon sets}: We know that, if $A$ is a convex set, then 
\[\abs{A}^{\frac{1}{2}+c}\ll S(A) \leq \abs{A}\]
for any constant $0<c<1/75$. 
The existence of convex sets with $S(A)=\abs{A}$ is trivial (for example, take any geometric progression). In this paper we construct a convex set $A$ such that $S(A) \ll \abs{A}^{3/4}$. We think that it could be possible to improve this upper bound to $S(A) \ll \abs{A}^{2/3}$, although we were unable to make this construction work (see the discussion after the proof of Theorem \ref{thm:Sidonupper} for more details). We believe that
\[ S(A) \gg \abs{A}^{2/3-o(1)}\]
for all convex $A$. 
\item \emph{Longest progression}: Let $P(A)$ be the length of the longest arithmetic progression in $A$. We know that, if $A$ is a convex set, then
\[2 \leq P(A) \ll \abs{A}^{1/2}.\]
The existence of convex sets with $P(A)=2$ is trivial (for example, take powers of $2$). In this paper we construct a convex set $A$ such that $P(A)\gg \abs{A}^{1/2}$ (see Theorem~\ref{thm:constructiongen}).
\item \emph{Largest structured subset}: Let $Q(A)$ be the maximal size of $B\subseteq A$ such that $\abs{B+B}\leq 100\abs{B}$ (of course there is nothing special about the constant $100$ here). We know that, if $A$ is a convex set, then
%\[Q(A) \ll \abs{A}^{123/150+o(1)}\]
%and conjecture that
\[Q(A) \ll\abs{A}^{2/3}.\]
Indeed, if $B$ has such small doubling then the pigeonhole principle implies that there is some $x \in B+B$ such that $r_{B+B}(x) \gg |B|$. However,
\[
|B| \ll r_{B+B}(x) \leq r_{A+A}(x) \ll |A|^{2/3},
\]
where the last inequality uses the fact that $A$ is convex and the aforementioned result of Schoen. In this paper we construct a convex set $A$ such that $Q(A)\gg \abs{A}^{1/2}$. We do not know where the truth lies.

It may be true that $Q(A)\gg \abs{A}^{2/3}$ infinitely often. Indeed, as far as we know, it is even possible that there exist arbitrarily large convex sets $A$ which contain sets $B$ of size $\asymp \abs{A}^{2/3}$ such that $B=P_1+P_2$ for some arithmetic progressions $P_1$ and $P_2$ (whence $\abs{B+B}\leq 4\abs{B}$).
\end{itemize}

In Section~\ref{sec-construction} we construct a convex set with many three-term arithmetic progressions, proving Theorem~\ref{thm:construction1}. In Section~\ref{sec-upper} we give a new proof of Schoen's upper bound $r_{A+A}$ (establishing Theorem~\ref{thm:3APupper}) and a construction showing it is best possible. In Section~\ref{sec-jarnik} we explain the connection to a well-known construction of Jarn\'ik. Finally, in Section~\ref{sec-sidon} we discuss Sidon sets in convex sets, proving Theorems~\ref{thm:Sidonlower} and \ref{thm:Sidonupper}.

\section{Constructing a convex set with additive structure}\label{sec-construction}
%Let $T_3(A)$ count the number of three-term arithmetic progressions inside $A$. Trivially we have
%\[\abs{A}\leq T_3(A)\leq \abs{A}^2.\]
%Since very structured sets (for example, arithmetic progressions) have $T_3(A)\gg \abs{A}^2$, the count $T_3(A)$ can be viewed as another measure of additive structure. To this end, it is natural to ask how it behaves for convex sets. Of course, one can take convex sets which do not contain any non-trivial three-term arithmetic progressions (for example, $A=\{1,2,4,\ldots,2^n\}$), and hence $T_3(A)=\abs{A}$ is certainly possible. The interesting question is how large $T_3(A)$ can possibly be. 

\begin{comment}

In this section we prove that
\[T_3(A) \leq \abs{A}^{3/2+o(1)},\]
and provide a construction showing that this is the best possible bound (up to $\abs{A}^{o(1)}$).

We begin with an easy upper bound.
\begin{theorem}
Let $A\subset \bbr$ be a convex set. We have
\[T_3(A) \leq \abs{A}^{5/3+o(1)}\]
and, conditional on Conjecture~\ref{conj1},
\[T_3(A) \leq \abs{A}^{3/2+o(1)}.\]
\end{theorem}
\begin{proof}
We can write 
\[T_3(A)=\sum_{c\in 2\cdot A}r_{A+A}(c),\]
where $r_{A+A}(x)$ counts the number of representations of $x=a+b$ where $a,b\in A$. By H\"{o}lder's inequality this is
\[\leq \abs{A}^{2/3}\brac{\sum_x r_{A+A}(x)^3}^{1/3}\leq \abs{A}^{5/3+o(1)},\]
where we have used the fact that
\[E_3(A)=\sum_x r_{A+A}(x)^3 \leq \abs{A}^{3+o(1)}.\]
For the second inequality, we may similarly bound
\[T_3(A) \leq \abs{A}^{1/2}E(A)^{1/2}\leq \abs{A}^{3/2+o(1)},\]
assuming Conjecture~\ref{conj1}.
\end{proof}

\end{comment}

The main goal of this section is to prove Theorem \ref{thm:construction1}, constructing large convex sets which contain many three-term arithmetic progressions. We derive this from the following more general result.

\begin{theorem}\label{thm:constructiongen}
For any $n, m\in \bbn$ with $n\geq 8m^2$ there is a convex set $A$ of size $\abs{A}=\Theta( n)$ which contains at least $ \frac{n}{4m}$ disjoint arithmetic progressions, each of length $m$.
\end{theorem}
\begin{proof}
For $m\leq \ell <2m$ let
\[f_\ell(x) = a x^2+\ell(x+n),\]
where $a=a_{m,n}>0$ is some parameter to be chosen later (all that matters is that it is strictly positive and sufficiently small depending only on $n$ and $m$). Note that, for each fixed $\ell$, $f_\ell(x)$ is a strictly convex function of $x$, and for each fixed $x$, $f_\ell(x)$ is a linear function of $\ell$. 

For each $m\leq \ell <2m$ let $1\leq y_\ell \leq \ell$ be an integer such that 
\[n+y_\ell\equiv 1\pmod{\ell},\]
so that there exists some integer $x_\ell\geq 1$ with
\[\ell x_\ell=n+(\ell+1)(y_\ell-1).\]
Note that
%\[x_\ell-y_\ell =\frac{n+y_\ell-1}{\ell}-\ell\geq \frac{n}{\ell}-\ell.\]
\[x_\ell-y_\ell =\frac{n+y_\ell-1}{\ell}-1,\]
and hence
\[
\frac{n}{\ell} -1 \leq x_\ell-y_\ell < \frac{n}{\ell}.
\]
%In particular, for all $m-1\leq\ell <2m$, provided $n\geq 4m^2$ this is $\asymp n/m$. 
Similarly, since $x_\ell-y_{\ell-1}=x_\ell-y_\ell + (y_{\ell}-y_{\ell-1})$ and the $y_{i}$ are not too far apart, we have
\[
\frac{n}{\ell} - \ell< x_\ell-y_{\ell-1}<\frac{n}{\ell} + \ell.
\]
%\textcolor{blue}{This is because $x_\ell-y_\ell\asymp n/m$ and $y_\ell-y_{\ell-1}=O(\ell)$ which is much smaller than $n/m$, so $x_\ell-y_{\ell-1}$ is still $\gg n/m$.}
We now let, again for $m\leq \ell < 2m$, 
\[B_\ell = \{ f_\ell(k) : y_{\ell-1} \leq k\leq x_\ell\}.\]
Note that, for this range of $\ell$, we have
\[
\frac{n}{4m} \leq |B_{\ell}| \leq \frac{2n}{m}.
\]
As the image of an interval of consecutive integers under a strictly convex function, each $B_\ell$ is a convex set. Moreover, we claim that the $B_\ell$ are pairwise disjoint, and furthermore that $A=\cup_{m\leq \ell <2m}B_\ell$ is also a convex set. Both of these facts follow from verifying the inequalities, for all $m\leq \ell<2m$,
\[f_\ell(x_\ell)-f_\ell(x_\ell-1)< f_{\ell+1}(y_\ell)-f_\ell(x_\ell)<f_{\ell+1}(y_{\ell}+1)-f_{\ell+1}(y_\ell).\]
Inserting the definition of $f$ these inequalities are equivalent to
\[a(2x_\ell-1)+\ell< a(y_\ell^2-x_\ell^2)+(\ell+1)y_\ell-\ell x_\ell+n< a(2y_\ell+1)+\ell+1.\]
By our choice of $x_\ell$ this simplifies to 
\[a(2x_\ell-1)-1< a(y_\ell^2-x_\ell^2)< a(2y_\ell+1).\]
The second inequality is true simply because $x_\ell >y_\ell>0$ and $a>0$. The former is true for sufficiently small $a>0$, as taking $a$ small enough ensures $a(2x_\ell-1)<1/2$ and $a(x_\ell^2-y_\ell^2)<1/2$, say. (In particular $a$ can be taken to be of the order $ m^2/n^2$.)

We have therefore produced a convex set $A$ of size $|A|=\Theta(n)$. Moreover, since $y_\ell \leq \ell\leq 2m$, each $B_\ell$ contains
\[\{ f_\ell (k) : 2m\leq k\leq \min_\ell x_\ell\},\]
and hence $A=\cup B_\ell$ contains, for each $2m\leq k\leq \min_\ell x_\ell$,
\[P_k = \{ f_\ell(k) : m\leq \ell <2m\},\]
which is an arithmetic progression of length $m$. Finally we note that the disjointness of the $B_\ell$ implies the disjointness of the $P_k$, and we are done (after noting that $\min_\ell x_\ell \geq n/(2m)$, and hence there are $\geq n/4m$ many such progressions in $A$).
\end{proof}

%\textcolor{red}{Proof from TB.tex will be here}

It is worth noting that the condition $n \gg m^2$ in the statement of Theorem \ref{thm:constructiongen} is crucial, and it is not possible for such a construction to exist with $n=o(m^2)$. This is because an arithmetic progression in a convex set $A$ can have length at most $O(|A|^{1/2})$ (as we prove below in Corollary~\ref{cor-prog}). 

Applying Theorem~\ref{thm:constructiongen} with $n=\Theta(m^2)$ gives the following corollary, which encompasses Theorem~\ref{thm:construction1} from the introduction.

\begin{corollary}
There are arbitrarily large convex sets $A\subset \bbr$ which contain $\gg \abs{A}^{1/2}$ many disjoint arithmetic progressions of length $\gg \abs{A}^{1/2}$.

In particular, for any fixed $k\geq 3$ there are arbitrarily large convex sets $A\subset \bbr$ which contain $\gg_k \abs{A}^{3/2}$ many non-trivial $k$-term arithmetic progressions.
\end{corollary}

%\textcolor{blue}{ORN: Jakob, you wanted to add something about homogeneous linear equations, I guess this would be a good place to insert that remark.}

\begin{remark}
    The same bound also holds for other translation invariant configurations of fixed size. For example, given any homogeneous linear equation $\sum_{i=1}^k a_ix_i=0$ with $\sum_{i=1}^k a_i=0$, the convex set $A$ constructed above yields $\gg \abs{A}^{3/2}$ many non-trivial solutions.
\end{remark}

\section{Additive representations for a convex set and upper bounding $T_3(A)$}\label{sec-upper}

The first goal of this section is to prove Theorem \ref{thm:3APupper}. The result is implicit in the work of Garaev \cite{Ga00}, and we have found a few slightly different proofs while working on this paper. Perhaps the easiest method is to use an upper bound for the number of representations of $A+A$. The following result can be taken from Schoen \cite{Sc14}.

\begin{theorem} \label{thm:reps}
For any convex set $A \subset \bbr$ and any $x \in \mathbb R$,
\[
r_{A+A}(x) \ll |A|^{2/3}.
\]
\end{theorem}

The proof in \cite{Sc14} is elementary. We give an alternative elementary proof, using the Erd\H{o}s-Szekeres Theorem.

\begin{theorem}[Erd\H{o}s-Szekeres \cite{ErSz35}] \label{thm:ES} %\textcolor{blue}{(not exactly, should we proof this version? TB: Isn't this exactly the usual statement? JF: Blue is the usual statement. That $r=s=\floor{\sqrt{m}})$ satisfies this can easily be checked, but that we can allow non-distinct elements requires some argument. TB: I checked, and fortunately this is actually already proved in the original paper of ES (i.e. the version allowing non-distinct elements). So we can just state this and quote the original paper and be done. JF: Okay, that's nice. Thank you!}] \label{thm:ES}
In any sequence of real numbers of length $m$ there exists a monotone subsequence of length $\floor{\sqrt{m}}$.

% \textcolor{blue}{Let $S$ be a sequence of distinct real numbers of length $m$ and let $r,s\in\mathbb{N}$ such that $m\geq(r-1)(s-1)+1.$ Then $S$ contains either a monotone increasing subsequence of length $r$ or a monotone increasing subsequence of length $s$.}

\end{theorem}

Removing such a subsequence from $S$ and applying this result again, and repeating this process, we have the following corollary.

\begin{corollary}\label{cor:ES}
Let $S$ be a sequence of real numbers of length $m$. There exist $t\gg m^{1/2}$ many disjoint monotone subsequences $U_1,\ldots,U_t$, each of length $\gg m^{1/2}$. 
\end{corollary}

We will also require the following property of convex sets.

\begin{lemma}\label{lmm:divide}
    Let $A\subset \bbr$ be a convex set and $\{a_1<b_1\leq a_2<b_2\leq\cdots \leq a_k<b_k\}\subseteq A$. If 
    \[b_{i+1}-a_{i+1}\leq b_i-a_i\]
    for $1\leq i<k$ then $$\left|A\cap \left(\bigcup_{i=1}^k \left(a_i,b_i\right)\right)\right|\geq\frac{k(k-1)}{2}.$$
In particular, $k\ll \abs{A}^{1/2}$.
\end{lemma}

\begin{proof}
    Let $n_i:=\lvert A\cap \left(a_i,b_i\right)\rvert$ for all $i\in[1,k]$. The largest distance between consecutive points of $A$ in $[a_i,b_i]$, denoted by $d_i^{\textrm{max}}$, is at least $\frac{b_i-a_i}{n_i+1}$, while the smallest distance between consecutive points of $A$ in $[a_{i+1},b_{i+1}]$, denoted by $d_{i+1}^{\textrm{min}}$, is at most $\frac{b_{i+1}-a_{i+1}}{n_{i+1}+1},$ for all $i\in[1,k-1]$. Since $A$ is convex 
    $$\frac{b_i-a_i}{n_i+1}\leq d_i^{\textrm{max}}<d_{i+1}^{\textrm{min}}\leq\frac{b_{i+1}-a_{i+1}}{n_{i+1}+1}\leq \frac{b_i-a_i}{n_{i+1}+1},$$ and therefore $n_i>n_{i+1}$. By induction it follows that, for $1\leq i\leq k$, we have $n_{i}\geq k-i$, and hence $$\left|A\cap \left(\bigcup_{i=1}^k \left(a_i,b_i\right)\right)\right|=\sum_{i=1}^k n_i\geq \sum_{i=1}^k (k-i)=\sum_{j=1}^{k-1} j=\frac{k(k-1)}{2}.$$
\end{proof}

We digress slightly to note the following consequence of Lemma~\ref{lmm:divide}, bounding the length of any arithmetic progression in a convex set, which follows taking $P=\{a_1<\cdots<a_{k+1}\}$ and $b_i=a_{i+1}$.
\begin{corollary}\label{cor-prog}
If $A\subset \bbr$ is a convex set and $P\subseteq A$ is an arithmetic progression then $\abs{P}\ll \abs{A}^{1/2}$.
\end{corollary}

%\textcolor{blue}{I don't know if we really should include this proof in the end. We could also just quote the result from Schoen and leave it at that. Or we could show how to derive it from Jarnik. There are too many choices, but perhaps we could try this suggestion and see what we think?}

%\textcolor{blue}{TB: Now that I've seen this proof I think it's a nice way to do it, arguably better than Schoen's presentation, so good to have it here I think. But should certainly mention link to Jarnik, and check whether there's already a proof of Jarnik using ES in the literature.}

%\textcolor{blue}{TB: I tried to tidy up the proof below, see what you think - original version is preserved as a comment below it. ORN: All looks very neat and tidy to me, nice one.}

We now return to the task of proving Theorem~\ref{thm:reps}, an upper bound for $r_{A+A}$.
\begin{proof}[Proof of Theorem \ref{thm:reps}]
Translating $A$ by $-x/2$ (which preserves the convexity of $A$) if necessary we can assume without loss of generality that $x=0$. Let $B:= A \cap (-A)\cap\mathbb{R}_{>0}$, so that $r_{A+A}(x)\leq  \abs{B}+1$. It suffices therefore to bound $m=\abs{B}$.

Let $B=\{x_1<x_2<\cdots<x_m\}$ and $\ell_i:=x_{i+1}-x_i$. By Theorem~\ref{thm:ES} there are $t\gg m^{1/2}$ many disjoint monotone subsequences $U_1,\ldots,U_t$ of $(\ell_1,\ldots,\ell_{m-1})$, each of length $\abs{U_i}\gg m^{1/2}$. By Lemma~\ref{lmm:divide}, if $U_i=(\ell_{u_1},\ldots,\ell_{u_k})$ is non-increasing then
\[\abs{A\cap \brac{\bigcup_{j=1}^{k}(x_{u_j},x_{u_{j}+1})}}\gg k^2 \gg m.\]
On the other hand, if $U_i$ is non-decreasing then
\[\abs{A\cap \brac{\bigcup_{j=1}^{k}(-x_{u_{j}+1},-x_{u_j})}}\gg k^2 \gg m.\]
In either case, if
\[A_i = A\cap\brac{ \brac{\bigcup_{j=1}^{k}(-x_{u_{j}+1},-x_{u_j})}\cup\brac{\bigcup_{j=1}^{k}(x_{u_j},x_{u_{j}+1})}},\]
then we have, for all $1\leq i\leq t$, the lower bound $\abs{A_i} \gg m$. It remains to note that the $A_i$ are disjoint (since all the intervals $(x_i,x_{i+1})$ are disjoint subsets of $\mathbb{R}_{>0}$), and so
\[\abs{A}\geq \sum_{i=1}^t\abs{A_i}\gg \abs{B}^{3/2},\]
and thus $\abs{B} \ll \abs{A}^{2/3}$ as desired.
\begin{comment}

$I^+_i:=(x_i,x_{i+1})$ and $I^-_i:=(-x_{i+1},-x_{i})$ for $i\in[1,m-1]$. Let $L_1:=(\ell_1,\ell_2,\cdots,\ell_{m-1})$, a sequence of length $r_1:=m-1$ Let $U_1=(\ell_{u_1},\ell_{u_2},\cdots,\ell_{u_{k_1}})$ be a monotonous subsequence of $L_1$ of length $k_1:=\floor{\sqrt{r_1}}$. Let $A_i:=A\cap\left(\bigcup_{i=1}^{k_1}(I_{u_i}^+\cup I_{u_i}^-)\right)$. If $U_1$ is non-increasing, then, using Lemma~\ref{lmm:divide},  $\bigcup_{i=1}^{k_1}I_{u_i}^+$ contains at least $\frac{k_1(k_1-1)}{2}$ many elements of $A$. Conversely, if $U_1$ is non-descreasing, then $\bigcup_{i=1}^{k_1}I_{u_i}^-$ contains at least $\frac{k_1(k_1-1)}{2}$ many elements of $A$, which gives us the combined bound $$|A_i|\geq\frac{k_1(k_1-1)}{2}.$$

We continue by setting $L_2:=L_1\setminus U_1$ and repeating this process for $M:=\floor{\frac{\sqrt{m-1}}{2}}$ many times and obtain $L_i$, $U_i$, $r_i$, $k_i$ and $A_i$ for $i\in [1,M]$. Note here that the $A_i$'s are pairwise disjoint.

Since $k_1=\min_{i\in[1,M]}\{k_i\}$, we get that $r_i\geq r_M\geq m-1-M\floor{\sqrt{k_1}}\geq \frac{m}{2}$, and therefore $k_i=\Theta(\sqrt{m})$ for all $i\in[1,M]$. Now $A$ contains at least $$\sum_{i=1}^M |A_i| \geq\sum_{i=1}^M \frac{k_i(k_i-1)}{2}=\sum_{i=1}^M \Theta(m)=\Theta(m^{\frac{3}{2}})$$ many elements.
\end{comment}
\end{proof}

Theorem \ref{thm:3APupper} now follows as a simple consequence of Theorem \ref{thm:reps}.

\begin{proof}[Proof of Theorem \ref{thm:3APupper}]
    We have
    \[
    T_3(A)= \sum_{x \in 2A}r_{A+A}(x) \ll |A|^{2/3}\sum_{x \in 2A}1 = |A|^{5/3}.
    \]
\end{proof}

Next, we will show that Schoen's upper bound for $r_{A+A}(x)$ is optimal. For this we need the following lemma, which shows that every sequence with non-decreasing distances can be transformed into a convex set by slightly moving the points between the non-increasing intervals.

\begin{lemma} \label{lmm:almostAPs}
Let $a<b<c$ be positive real numbers and $B:=\{x_0<x_1<\cdots<x_k\}$ be an increasing sequence of real numbers such that
$$
x_{i}-x_{i-1}=
\begin{cases}
a & \text{if } i= 1,\\
b & \text{if } 1<i<k\textrm{, and}\\
c & \text{if } i= k.
\end{cases}
$$
There exists a convex set $C$ of size $|B|$ which contains $x_0$, $x_1$, $x_{k-1}$, and $x_k$.
\end{lemma}

\begin{proof}
Let $m:=\min\{b-a,c-b\}$ and consider the set
$$
D=
\begin{cases}
\left\{b-\frac{m}{k}i\;|\;{i\in [-\frac{k-3}{2},\frac{k-3}{2}]}\right\} & \text{if } k \text{ is odd and }\\
\left\{b-\frac{m}{k}i\;|\;{i\in [-\frac{k-2}{2},\frac{k-2}{2}]\setminus\{0\}}\right\} & \text{if } k \text{ is even}.
\end{cases}
$$
Let $\{d_2<d_3<\cdots<d_{k-1}\}$ be $D$ written in increasing order and let $C:=\{y_0=x_0<y_1 <\cdots <y_k\}$ be the increasing sequence of real numbers, such that their gaps satisfy
$$
y_{i}-y_{i-1}=
\begin{cases}
a & \text{if } i= 1,\\
d_i & \text{if } 1<i<k\textrm{, and}\\
c & \text{if } i= k.
\end{cases}
$$
Since $y_1-y_0=a$ this set contains $x_0$ and $x_1$. Furthermore, since
\[y_k-y_0=a+\sum_{i=2}^{k-1}d_i+c=a+(k-2)b+c=x_k-x_0,\]
this set also contains $x_k$, and so since $y_k-y_{k-1}=c$ it also contains $x_{k-1}$. Finally, this set is clearly convex, since
\[a\leq b-m<d_2<d_3<\cdots <d_{k-1}<b+m\leq c.\]
\end{proof}

\begin{theorem} \label{thm:3APLower}
There exist arbitrarily large finite convex sets $A\subseteq\mathbb{R}$ and $x\in A$ such that $x$ is the centre of $\Omega(|A|^{2/3})$ many $3$-term arithmetic progressions.
\end{theorem}
\begin{proof}
Let $n\in\mathbb{N}$ and 
\[Q:=\left\{\frac{a}{b} : 1\leq a<b\leq n\right\}.\]
Note that $m:=|Q|=\Theta(n^2)$. We order the elements of $Q$ as 
\[0<\frac{a_1}{b_1}<\frac{a_2}{b_2}<\cdots<\frac{a_m}{b_m}<1,\]
so that
$$\frac{a_i}{a_{i+1}}<\frac{b_i}{b_{i+1}}.$$ We can therefore inductively choose a sequence of real numbers $\ell_1,\ldots,\ell_m$ with $\ell_1=1$ and 
\begin{align}
\frac{a_i}{a_{i+1}}<\frac{\ell_i}{\ell_{i+1}}<\frac{b_i}{b_{i+1}}.\label{eqn:ell}
\end{align}

Let $B:=\{x_1<\cdots<x_m<x_{m+1}<x_{m+2}<\cdots<x_{2m+1}\}$ be an increasing sequence of real numbers such that
$$
x_{i+1}-x_{i}=
\begin{cases}
\ell_i & \text{if } i\leq m\textrm{ and}\\
\ell_{2m-i+1} & \text{if } i\geq m+1.
\end{cases}
$$
Define $x:= x_{m+1}$. Since $B$ is symmetric around $x$ (that is, $x_i+x_{2m+1-i}=x$ for $1\leq i\leq m$) we know that $r_{B+B}(2x)\geq m$.

We will construct a convex set $A\supseteq B$. For $i\leq m$ we add $b_i-1$ points in the interval $(x_{i+1},x_i)$ and $a_i-1$ points in the interval $(x_{2m-i+1},x_{2m-i+2})$. Note that both of these intervals are of length $\ell_i$. We choose these points such that the intervals are split into equal parts of length $\frac{\ell_i}{b_i}$ and $\frac{\ell_i}{a_i}$, respectively. By \eqref{eqn:ell} we have, for $1\leq i\leq m$,
\[\frac{\ell_i}{b_i}<\frac{\ell_{i+1}}{b_{i+1}}\textrm{  and  }\frac{\ell_{i+1}}{a_{i+1}}<\frac{\ell_{i}}{a_{i}}.\]
These inequalities, coupled with the fact that $\frac{\ell_m}{b_m}<\frac{\ell_m}{a_m}$, ensure that the consecutive distances are non-decreasing. Using Lemma~\ref{lmm:almostAPs} we can slightly move the newly-added points to obtain a convex set $A$ containing $B$. Finally, since $|A|=|B|+\sum_{i=1}^m(a_i-1)+\sum_{i=1}^m(b_i-1)\leq |B|+ O(mn)=O(n^3)$ we have $$r_{A+A}(2x)\geq r_{B+B}(2x)\geq m\gg |A|^{2/3}.$$

\end{proof}

\begin{corollary}
There exist  finite convex sets $A\subseteq\mathbb{R}$ and $x\in\mathbb{R}$ such that $r_{A+A}(x)\gg |A|^{2/3}$.
\end{corollary}

Curiously, the problem of upper bounding the number of representations of an element of the difference set determined by a convex set has a very different answer. Schoen \cite{Sc14} gave a construction of a convex set with $r_{A-A}(x) \gg |A|$ for some $x \neq 0$. The best possible multiplicative constant for this problem was precisely determined by Roche-Newton and Warren \cite{RoWa22}.

\section{A connection to Jarn\'ik's theorem}\label{sec-jarnik}

The following reformulation of a classical result of Jarn\'{i}k \cite{Ja26} answers the question of how many grid points a convex curve can contain.

\begin{theorem}[Jarn\'ik \cite{Ja26}] \label{thm:jrnk}
    Let $f$ be a strictly convex function and $n\in\mathbb{N}$. The graph of $f$ intersects the integer grid $[n]^2$ in $O(n^{2/3})$ points.
\end{theorem}

In the same paper, Jarn\'{i}k showed that this result is optimal. That is, a construction was given of a convex curve containing $\Omega(n^{2/3})$ points from $[n]^2$. The purpose of this section is to sketch how this same construction can also be used to deduce our earlier Theorem \ref{thm:3APLower}. We begin by explaining Jarn\'{i}k's construction, and one may immediately observe some similarities between this and the construction in the proof of Theorem \ref{thm:3APLower}.

%We can obtain the construction of the proof of Theorem~\ref{thm:3APLower} by a slightly modified construction of  Jarn\'ik's lower bound.

%In this section we provide a short proof sketch of both theorems.

Let $n\in\mathbb{N}$ and let $Q:=\{\frac{a}{b} : 1\leq a, b\leq n\}\setminus \{1\}$. The exclusion of $1$ is necessary to get one central point $x$ for Theorem~\ref{thm:3APLower} and it will not change the asymptotics for Jarn\'ik's bound. As before we order $Q$ as
$$\frac{1}{n}=\frac{a_{-m}}{b_{-m}}<\frac{a_{-m+1}}{b_{-m+1}}<\cdots<\frac{a_{-1}}{b_{-1}}<1<\frac{a_{1}}{b_{1}}<\cdots<\frac{a_{m}}{b_{m}}=\frac{n}{1}.$$ 
Note that, by the symmetry of the ratio set, we have $a_i=b_{-i}$ for all $-n  \leq i \leq n$.

We then construct a set of points $\mathcal{X:=\{}X_{-m},X_{-m+1},\cdots,X_m\}$ by setting $X_0:=(0,0)$, $X_i-X_{i-1}=(a_i,b_i)$ and $X_{-i+1}-X_{-i}=(a_{-i},b_{-i})$ for $i\in[1,m]$. Let $f$ be any convex curve connecting $\mathcal{X}.$  $\mathcal{X}$ is a set of size $2m+1=\Theta(n^2)$ and is contained in the grid $[-mn,mn]^2$. Therefore, $f$ and $\mathcal{X}$ show the sharpness of Theorem~\ref{thm:jrnk}.

Now let $\mathcal{Y}$ be the set of points on $f$ having integral $y$-coordinate in the range $[-mn,mn]$. In particular $\mathcal{X}\subseteq\mathcal{Y}$. We map $\mathcal{Y}$ to the real line via  $\pi:(x,y)\mapsto x+y$ to obtain the convex set $\mathcal{Y}_0.$
For visual purposes it is easier to instead think of projecting onto the line $\{y=x\}$. This projection is given by $$(x,y)\mapsto\left(\frac{x+y}{2},\frac{x+y}{2}\right),$$ and is therefore just a rotated and stretched variant (see Figure~\ref{fig:jrnk}).
$\mathcal{Y}_0$ is a set of size $\leq 2mn+1$ and it contains the arithmetic progressions $(\pi(X_{-i}),\pi(X_0),\pi(X_i))$ for $i\in[1,m]$. This shows Theorem~\ref{thm:3APLower}, moreover, it is even a special case of the construction in its proof, where $$\frac{a_i}{a_{i+1}}<\frac{\ell_i}{\ell_{i+1}}=\frac{a_i+b_i}{a_{i+1}+b_{i+1}}<\frac{b_i}{b_{i+1}}.$$

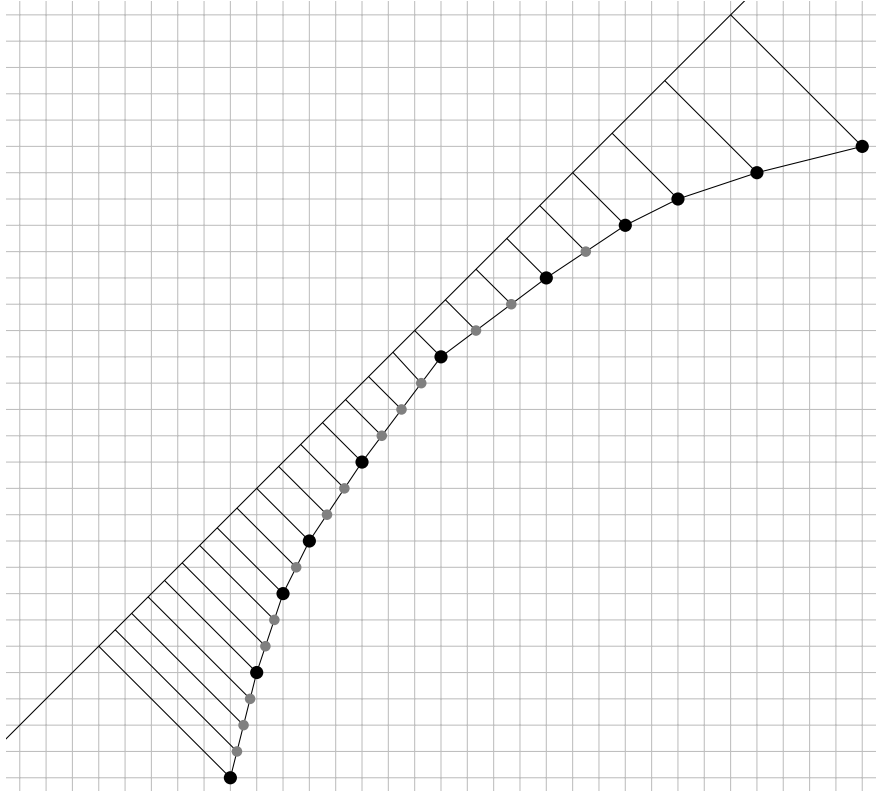
\begin{figure}[H]
\centering
\begin{tikzpicture}[yscale=1,scale=0.35]

\clip (-16.5,-16.5) rectangle + (33,30);

\draw[step=1.0,lightgray,thin,xshift=-0.0cm,yshift=-0.0cm] (-20,-20) grid (20,20);

\draw (-18,-16) -- (14,16);
\fill (0,0) circle (0.25);
\fill (4,3) circle (0.25);
\fill (-3,-4) circle (0.25);
\fill (7,5) circle (0.25);
\fill (-5,-7) circle (0.25);
\fill (9,6) circle (0.25);
\fill (-6,-9) circle (0.25);
\fill (12,7) circle (0.25);
\fill (-7,-12) circle (0.25);
\fill (16,8) circle (0.25);
\fill (-8,-16) circle (0.25);

\draw (-8,-16) -- (-7,-12);
\draw (-7,-12) -- (-6,-9);
\draw (-6,-9) -- (-5,-7);
\draw (-5,-7) -- (-3,-4);
\draw (-3,-4) -- (0,0);
\draw (0,0) -- (4,3);
\draw (4,3) -- (7,5);
\draw (7,5) -- (9,6);
\draw (9,6) -- (12,7);
\draw (12,7) -- (16,8);

\draw (0,0) -- (-1,1);
\draw (4,3) -- (2.5,4.5);
\draw (-3,-4) -- (-4.5,-2.5);
\draw (7,5) -- (5,7);
\draw (-5,-7) -- (-7,-5);
\draw (9,6) -- (6.5,8.5);
\draw (-6,-9) -- (-8.5,-6.5);
\draw (12,7) -- (8.5,10.5);
\draw (-7,-12) -- (-10.5,-8.5);
\draw (16,8) -- (11,13);
\draw (-8,-16) -- (-13,-11);

\draw (-7.75,-15) -- (-12.375,-10.375);
\draw (-7.5,-14) -- (-11.75,-9.75);
\draw (-7.25,-13) -- (-11.125,-9.125);
\draw (-6.67,-11) -- (-9.83,-7.83);
\draw (-6.33,-10) -- (-9.167,-7.167);
\draw (-5.5,-8) -- (-7.75,-5.75);
\draw (-4.33,-6) -- (-6.167,-4.167);
\draw (-3.67,-5) -- (-5.33,-3.33);
\draw (-2.25,-3) -- (-3.625,-1.625);
\draw (-1.5,-2) -- (-2.75,-0.75);
\draw (-0.75,-1) -- (-1.825,0.175);
\draw (1.33,1) -- (0.167,2.167);
\draw (2.67,2) -- (1.33,3.33);
\draw (5.5,4) -- (3.75,5.75);

\fill[gray] (-7.75,-15) circle (0.2);
\fill[gray] (-7.5,-14) circle (0.2);
\fill[gray] (-7.25,-13) circle (0.2);
\fill[gray] (-6.67,-11) circle (0.2);
\fill[gray] (-6.33,-10) circle (0.2);
\fill[gray] (-5.5,-8) circle (0.2);
\fill[gray] (-4.33,-6) circle (0.2);
\fill[gray] (-3.67,-5) circle (0.2);
\fill[gray] (-2.25,-3) circle (0.2);
\fill[gray] (-1.5,-2) circle (0.2);
\fill[gray] (-0.75,-1) circle (0.2);
\fill[gray] (1.33,1) circle (0.2);
\fill[gray] (2.67,2) circle (0.2);
\fill[gray] (5.5,4) circle (0.2);

\end{tikzpicture}
\caption{An illustration showing how the construction of Theorem \ref{thm:3APLower} can be derived from Jarn\'{i}k's construction of a convex curve with many lattice points.}\label{fig:jrnk}
\end{figure}

\section{Sidon sets in convex sets}\label{sec-sidon}

We now turn to the problem of finding large Sidon sets in convex sets. We first prove Theorem \ref{thm:Sidonlower}, using a standard probabilistic argument.

\begin{proof}[Proof of Theorem \ref{thm:Sidonlower}]
    Let $A'\subseteq A$ be a $p$-random subset, where we include each $x\in A$ independently with probability $p$. The value of $0<p<1$ will be specified later. The expected size of $A$ is $p\abs{A}$.
    There are two kinds of non-trivial energy solutions that we need to consider separately, namely the solutions to the equation
    \[
    a+b=c+d
    \]
    where the values $a,b,c,d$ are pairwise distinct, and the solutions for which either $a=b$ or $c=d$. Let $T$ denote the number of solutions of the latter type and note that $T$ is at most a constant factor times the number of non-trivial $3$-term arithmetic progressions in $A$. Let $E'$ denote the number of solutions of the first type and note that $E' \leq E(A)$. Therefore, the expected number of non-trivial additive quadruples is, for any $\epsilon>0$,
\[p^3T+p^4E \ll p^3T_3(A)+p^4E(A) \ll_\epsilon p^3|A|^{5/3}+p^4|A|^{\frac{123}{50}+\epsilon},
\]
where we have used \eqref{energyupper} and Theorem \ref{thm:3APupper}. We choose $p=c\abs{A}^{-\frac{73}{150}-\epsilon}$ for some sufficiently small constant $c=c(\epsilon)>0$, so that the expected number of non-trivial additive quadruples in $A$ is at most
\[\frac{p|A|}{2} = \tfrac{1}{2}\mathbb E |A'|.\]
In particular, if we let $Q$ denote the number of non-trivial energy solutions in $A'$, then
\[
\mathbb E(|A'| - Q)\geq \frac{p|A|}{2}.
\]
It follows that there exists a set $A' \subseteq A$ such that $|A'| - Q \geq \frac{p|A|}{2} $. Therefore, after pruning the set $A'$ by deleting an element for each contribution to $Q$, we find $A'' \subseteq A' \subseteq A$ such that $A''$ is Sidon and 
\[
|A''| \geq |A'| - Q \geq \frac{p|A|}{2} \gg_\epsilon |A|^{\frac{77}{150} - \epsilon}.
\]
\end{proof}

We note that using the energy bound from Conjecture \ref{conj1}, rather than \eqref{energyupper}, in this proof yields the existence of a Sidon set $A' \subseteq A$ with $|A'| \geq |A|^{2/3- o(1)}$.

Finally, we turn to the task of proving Theorem \ref{thm:Sidonupper}. That is, we seek to show the existence of a convex set $A$ such that every set $A' \subseteq A$ with $|A'| \gg |A|^{3/4}$ contains a non-trivial additive energy solution. For this, we appeal to a graph theoretical argument Erd\H{o}s \cite{Er84} used in order to establish the existence of a set with very small additive energy but still having fairly small $S(A)$. First, we formally state a previously mentioned result of Ruzsa and Zhelezov \cite{RuZh19} that will be used in the proof.

\begin{theorem} \label{thm:RZ}
Let $ n$ be a sufficiently large integer. There are sets $B,C\subseteq \bbr$ with $\abs{B} \asymp n$ and $\abs{C}\asymp n$ such that $B+C$ contains a convex set of size $\asymp n^2$. 
\end{theorem}

\begin{proof}[Proof of Theorem \ref{thm:Sidonupper}]
Apply Theorem \ref{thm:RZ}. We obtain sets $B$ and $C$, with $\abs{B} \asymp n$ and $\abs{C} \asymp n$, such that $B+C$ contains a convex set $A$ with $ |A| \gg n^2$. We can use the set $A$ to build a bipartite graph on $B \times C$. Namely, for each $a \in A$, fix a representation $a=b+c$ arbitrarily, with $(b,c) \in B \times C$. For each such $a$, draw the edge connecting $b$ and $c$.

Suppose that $A' \subset A$ is a Sidon set with size $Kn^{3/2}$ for a sufficiently large constant $K$. Then the subgraph corresponding to $A'$ contains $Kn^{3/2}$ edges, and it follows from the K\H{o}v\'{a}ri-S\'{o}s-Tur\'{a}n Theorem that this graph contains a $C_4$. 
%If $A'\subseteq A$ is Sidon and has size $\gg n^2$ then the induced graph $G\subseteq B\times C$ has $\gg n^2$ edges, and hence must contain a $C_4$. 

In other words, there exist $b_1,b_2\in B$ and $c_1,c_2\in C$ such that 
\[b_1+c_1,b_1+c_2,b_2+c_1,b_2+c_2\in A',\]
which contradicts $A'$ being Sidon, since
\[(b_1+c_1)+(b_2+c_2)=(b_1+c_2)+(b_2+c_1).\]

\end{proof}

As a final remark, let us discuss a potential improvement to Theorem \ref{thm:Sidonupper}. If we had an unbalanced version of Theorem \ref{thm:RZ}, namely with $|C| \asymp n^2$ and a convex set in $B+C$ with size $\asymp n^3$, then the same argument as above (applying a version of the K\H{o}v\'{a}ri-S\'{o}s-Tur\'{a}n Theorem for an unbalanced bipartite graph) could be used to give a construction with $S(A) \ll |A|^{2/3}$. We were, however, unable to prove an unbalanced form of the Ruzsa-Zhelezov result.

\subsection*{Acknowledgements} We are grateful to Brandon Hanson, Michalis Kokkinos, Akshat Mudgal, Misha Rudnev, Ilya Shkredov, and Audie Warren for helpful conversations. Jakob Führer and Oliver Roche-Newton were supported by the Austrian Science Fund FWF Project PAT2559123. Thomas Bloom is supported by a Royal Society University Research Fellowship.

\bibliographystyle{plain}
\bibliography{refs}

\end{document}